\documentclass[11pt]{amsart}

\usepackage{epigamath}

\usepackage[english]{babel}

\numberwithin{equation}{section}

\usepackage{enumitem}

\newtheorem{theorem}{Theorem}[section]

\newtheorem{lemma}[theorem]{Lemma}
\newtheorem{proposition}[theorem]{Proposition}

\theoremstyle{definition}

\newtheorem{definition}[theorem]{Definition}

\theoremstyle{remark}
\newtheorem{remark}[theorem]{Remark}

\newcommand{\del}{\partial}
\newcommand{\delbar}{\overline{\partial}}
\newcommand{\ddbar}{\del\delbar}
\newcommand{\vp}{\varphi}
\newcommand{\ve}{\varepsilon}

\newcommand{\bS}{\mathbb{S}}

\DeclareMathOperator{\rM}{M}
\newcommand{\abs}[1]{\left\vert #1 \right\vert}
\newcommand{\set}[1]{\left\{ #1 \right\}}
\newcommand{\norm}[1]{\left\Vert #1 \right\Vert}
\newcommand{\romfigure}[1]{\MakeUppercase{\romannumeral #1}}

\EpigaVolumeYear{6}{2022} \EpigaArticleNr{12} \ReceivedOn{February 7,
2022}
\InFinalFormOn{February 8, 2022}
\AcceptedOn{February 9, 2022}

\title{A gluing construction of projective K3 surfaces}
\titlemark{A gluing construction of projective K3 surfaces}

\author{Takayuki Koike}
\address{Department of Mathematics, Graduate School of Science, Osaka City University, 3-3-138 Sugimoto, Osaka 558-8585, Japan}
\email{tkoike@osaka-cu.ac.jp}

\author{Takato Uehara}
\address{Department of Mathematics, Faculty of Science, Okayama University, 1-1-1, Tsushimanaka, Okayama, 700-8530, Japan}
\email{takaue@okayama-u.ac.jp}

\authormark{T.~Koike and T.~Uehara}

\AbstractInEnglish{We construct a non-Kummer projective K3 surface $X$ which admits compact Levi-flats by holomorphically patching two open complex surfaces obtained as the complements of tubular neighborhoods of elliptic curves embedded in blow-ups of the projective plane at nine general points.}



\acknowledgement{The first author is supported by JSPS KAKENHI Grant Number 20K14313 and by the program Leading Initiative for Excellent Young Researchers (LEADER, No.~J171000201). 
The second author is supported by JSPS KAKENHI Grant Number 19K03544.}

\begin{document}



\maketitle

\begin{prelims}

\DisplayAbstractInEnglish

\bigskip

\DisplayKeyWords

\medskip

\DisplayMSCclass







\end{prelims}


\newpage

\setcounter{tocdepth}{1}

\tableofcontents


\section{Introduction}\label{section:introduction}

In the paper \cite{KU}, we gave a method, the so-called {\it gluing method}, for constructing a family of K3 surfaces, that is, 
we constructed such a K3 surface by holomorphically gluing two open complex surfaces obtained as the complements of tubular neighborhoods of elliptic curves embedded in blow-ups of the projective planes at nine points. 
The family has complex dimension $19$ and each K3 surface of the family admits compact Levi-flat hypersurfaces. 
In this paper, we will show that there are {\it projective} K3 surfaces among the family. 
One of the main results is given as follows: 

\begin{theorem}\label{thm:main}
There exists a deformation $\pi\colon\mathcal{X}\to B$ of projective K3 surfaces over an {$18$ dimensional} complex manifold $B$ with injective Kodaira-Spencer map 
such that each fiber $X_b:=\pi^{-1}(b)$ admits a holomorphic immersion $F_b\colon \mathbb{C}\to X_b$ with the property that  
the Euclidean closure of the image $F_b(\mathbb{C})$ in $X_b$ is a compact real analytic hypersurface $C^\omega$-diffeomorphic to a real $3$-dimensional torus $\bS^1\times \bS^1\times \bS^1$ which is Levi-flat. 
Especially, $F_b(\mathbb{C})$ is Zariski dense in $X_b$ whereas it is not Euclidean dense. 
Moreover, $X_b$ is non-Kummer for almost every $b\in B$ in the sense of the Lebesgue measure. 
\end{theorem}

In the construction of  K3 surfaces given in the paper \cite{KU}, we prepare two surfaces $S^+$ and $S^-$ obtained from the blow-ups of 
the projective plane $\mathbb{P}^2$ at nine points $\{p_1^{\pm}, \ldots, p_9^{\pm}\}$ with smooth elliptic curves $C^{\pm} \in |K_{S^\pm}^{-1}|$. 
Here we assume that $(S^{\pm},C^{\pm})$ satisfy the following two conditions:
\begin{enumerate}[label=(\alph*)]
\item there exists an isomorphism $g : C^+ \to C^-$ such that $g^*N_- \cong N_+$, where $N_\pm:=N_{C^\pm/S^\pm}$ are the normal bundles of $C^{\pm}$ in $S^{\pm}$, and 
\item the normal bundles $N_\pm \in \mathrm{Pic}^0(C^{\pm})$ satisfy the Diophantine condition (see Definition~\ref{def:diophnormal}). 
\end{enumerate}
Then Arnol'd's theorem \cite{A} guarantees that there exist {\it analytically linearizable neighborhoods} $W^{\pm} \subset S^{\pm}$ of $C^{\pm}$ in $S^{\pm}$, namely, 
$W^{\pm}$ are tubular neighborhoods of $C^{\pm}$ in $S^{\pm}$ which are biholomorphic to neighborhoods of the zero sections in $N_{\pm}$. 
In other words, there exist a pair $(p,q) \in \mathbb{R}^2$ that satisfies the Diophantine condition (see Definition \ref{def:diophpair}) and a positive real number $R>1$ such that 
$W^{\pm}$ are expressed as 
\begin{equation} \label{eqn:isomorphism}
W^\pm \cong \set{(z^\pm, w^\pm)\in \mathbb{C}^2\mid \abs{w^\pm}<R }/\sim_\pm, 
\end{equation}
where $\sim_\pm$ are the equivalence relations generated by
\[
(z^\pm, w^\pm) \sim_\pm 
(z^\pm+1,\ \exp(\pm p\cdot 2\pi\sqrt{-1})\cdot w^\pm) \sim_\pm 
(z^\pm+\tau,\ \exp(\pm q\cdot 2\pi\sqrt{-1})\cdot w^\pm) 
\]
with $\tau\in \mathbb{H}:=\set{ \tau \in\mathbb{C}\mid {\rm Im}\,\tau>0}$ (here note that $C^+ \cong C^-$ via $g$). 
From now on, we fix $(p,q)$, $(S^{\pm},C^{\pm})$, $g$, and isomorphisms \eqref{eqn:isomorphism}. 

In the present paper, we take an appropriate $\xi \in \mathbb{C}$ and consider $g_\xi:=\ell_\xi\circ g$, where $\ell_{\xi} : C^- \cong \mathbb{C}/\langle 1,\tau \rangle \circlearrowleft$ is the translation induced from 
$\mathbb{C}\ni z\mapsto z+\xi\in \mathbb{C}$. Note that 
$g_\xi^*N_-\cong N_+$ remains true since $N_{\pm} \in \mathrm{Pic}^0(C^{\pm})$. 
For each $s \in \Delta:=\set{s\in\mathbb{C}\mid |s|<1}$ with $s \neq 0$, we define open submanifolds $M_s^\pm$ of $S^{\pm}$ by
\[
M_s^\pm := S^\pm\setminus \left\{[(z^\pm, w^\pm)]\in W^\pm\mid \abs{w^\pm}\leq \sqrt{\abs{s}}/R\right\}, 
\]
which contain
\[
V_s^\pm := \set{[(z^\pm, w^\pm)]\in W^\pm\mid \sqrt{|s|}/R<|w^\pm|< \sqrt{|s|}R }
\]
as neighborhoods of boundaries of $M_s^\pm$, and a biholomorphism $f_s\colon V_s^+\to V_s^-$ by 
\[
f_s\left([(z^+, w^+)]\right)=\left[(g_\xi(z^+),s/w^+)\right]. 
\]
Then by identifying $V_s^+$ and $V_s^-$ via the biholomorphic map $f_s$, we can patch $M_s^+$ and $M_s^-$ to define a compact complex surface $X_s$. 
In the paper  \cite{KU}, we showed that $X_s$ is a K3 surface and that the nowhere vanishing holomorphic $2$-form $\sigma_s$ on $X_s$ satisfies 
\[
\sigma_s|_{V_s}=c\cdot \frac{dz\wedge dw}{w}
\]
for some $c \in \mathbb{C}^*$, where $V_s \subset X_s$ is the open submanifold corresponding to $V_s^+\cong V_s^-$ and $(z,w)$ are the coordinates induced from $(z^+,w^+)$. 

For each $\xi$, these K3 surfaces $X_s$ with $s \in \Delta \setminus \{0\}$ are the fibers of a proper holomorphic map 
\[
\mathcal{X}\to \Delta
\]
from a smooth complex manifold $\mathcal{X}(=\mathcal{X}(\xi))$ such that 
\begin{itemize}
\item[---] each fiber over $s \in \Delta \setminus \{0\}$ coincides with the K3 surface $X_s$, 
\item[---] the fiber $X_0$ over $0 \in \Delta$ is a compact complex variety with normal crossing singularities whose irreducible components are $S^+$ and $S^-$ and whose singular part is the one obtained by identifying $C^+$ and $C^-$ via $g_{\xi}$, and thus
\item[---] $\mathcal{X}\to \Delta$ is a type \romfigure{2} degeneration of K3 surfaces (see Section~\ref{section:proof_main_i}). 
\end{itemize}

We notice that $V_s \subset X_s$ is biholomorphic to a topologically trivial annulus bundle over the elliptic curve $C:=C^+ \cong C^-$, 
and hence homotopic to $\bS_{\alpha}^1 \times \bS_{\beta}^1 \times \bS_{\gamma}^1$, where $\bS_{\alpha}^1$ and $\bS_{\beta}^1$ are circles in $V_s$ such that $\bS_{\alpha}^1 \times \bS_{\beta}^1$ is a $C^\infty$ section of the bundle, 
and $\bS_{\gamma}^1$ is a circle in a fiber of the bundle which generates the fundamental group. Then we define the $2$-cycles $A_{\alpha\beta}$, $A_{\beta\gamma}$, $A_{\gamma\alpha}$ by
\[
A_{\alpha\beta}=\bS_{\alpha}^1 \times \bS_{\beta}^1, \quad A_{\beta\gamma}=\bS_{\beta}^1 \times \bS_{\gamma}^1, \quad\text{and}\quad  A_{\gamma\alpha}=\bS_{\gamma}^1 \times \bS_{\alpha}^1. 
\]
In addition to the $2$-cycles $A_{\alpha\beta}$, $A_{\beta\gamma}$, $A_{\gamma\alpha}$, 
each K3 surface $X_s$ admits a marking, which gives $22$ generators of the second homology group $H_2(X_s, \mathbb{Z})$ denoted by 
\begin{equation} \label{eqn:generator}
A_{\alpha\beta},\ A_{\beta\gamma},\ A_{\gamma\alpha},\ B_\alpha,\ B_\beta,\ B_\gamma,\ C_{12}^+,\ C_{23}^+,\ \ldots, C_{78}^+,\ C_{678}^+,\ C_{12}^-,\ C_{23}^-,\ \ldots,\ C_{78}^-,\ C_{678}^-.
\end{equation}
In \S \ref{sec:calculation}, we will give the definitions of these generators. 

Now let $L^{\pm}$ be holomorphic line bundles on $S^{\pm}$ with $(L^+\cdot C^+)=(L^-\cdot C^-)$. 
Assume that there exists $\xi \in \mathbb{C}$ such that $g_\xi^*\left(L^-|_{C^-}\right)\cong L^+|_{C^+}$. 
Note that such a $\xi$ always exists when $(L^+\cdot C^+)=(L^-\cdot C^-)\not=0$. 
We fix such a $\xi \in \mathbb{C}$, and consider the deformation family $\mathcal{X}\to \Delta$. 

\begin{theorem}\label{thm:main1}
Under the above setting, we have the following. 
\begin{enumerate}[label=(\roman*)]
\item For any $s \in \Delta$, the line bundles $L^+|_{M_s^+}$ and $L^-|_{M_s^-}$ glue to define a holomorphic line bundle $L_s=L^+ \vee L^-$ on $X_s$. 
Moreover there exists a holomorphic line bundle $\mathcal{L}\to \mathcal{X}$ such that $\mathcal{L}|_{X_s}=L_s$ for each $s \in \Delta$.
\item If $L^\pm$ are ample, then there exists $\ve_0>0$ such that $L_s$ is ample for any $s \in \Delta$ with $0<|s|<\ve_0$.
\item Let $L$ be a holomorphic line bundle on $X_s$ for some $s \in \Delta \setminus \{0 \}$. Then the following are equivalent.
\begin{enumerate}
\item There exist line bundles $L^{\pm}$ on $S^{\pm}$ with $(L^+\cdot C^+)=(L^-\cdot C^-)$ such that $L=L^+\vee L^-$.
\item There exists a line bundle $\mathcal{L}\to \mathcal{X}$ such that $L=\mathcal{L}|_{X_s}$. 
\item $(L\cdot A_{\beta\gamma})=(L\cdot A_{\gamma\alpha})=0$. 
\end{enumerate}
\end{enumerate}
\end{theorem}

In our arguments it is important  to describe the line bundles on $V_s^{\pm}$ and on $W^{\pm}$, which is given in Section~\ref{sec:linebundles} after preliminary studies in Section~\ref{sec:preliminaries}. 
Then we will prove the main theorems in Section~\ref{sec:proofs}. 
Moreover, we will determine the Chern class $c_1(L_s)$ of the line bundle $L_s$ in terms of the marking \eqref{eqn:generator} in Section~\ref{sec:calculation}.

\subsection*{Acknowledgments} 
The authors would like to give special thanks to Prof. Takeo Ohsawa and Prof. Yuji Odaka whose enormous supports and insightful comments were invaluable during the course of their study. 


\section{Preliminaries} \label{sec:preliminaries}

\subsection{Neighborhoods of elliptic curves} \label{subsec:nbhdellcurve}

First we give the following definition. 

\begin{definition} \label{def:diophpair}
Let $(p,q) \in \mathbb{R}^2$ be a pair of real numbers. 
\begin{enumerate}
\item $(p,q)$ is called a {\it torsion pair} if $(p,q) \in \mathbb{Q}^2$. Otherwise, $(p,q)$ is called a {\it non-torsion pair}. 
\item $(p,q)$ is said to satisfy the {\it Diophantine condition} if there exist $\alpha>0$ and $A>0$ such that 
\[
\min_{\mu,\nu \in \mathbb{Z}} \abs{n(p+q \sqrt{-1})-(\mu+\nu \sqrt{-1})} \ge A \cdot n^{-\alpha}
\]
for any $n \in \mathbb{Z}_{>0}$.
\end{enumerate}
\end{definition}
Of course, if $(p,q)$ satisfies the Diophantine condition, then $(p,q)$ is a non-torsion pair. 

Let $X$ be a complex manifold. Denote by ${\rm Pic}(X)$ the Picard group of $X$, the group of isomorphism classes of holomorphic line bundles on $X$, 
and by ${\rm Pic}^0(X)$ the subgroup of ${\rm Pic}(X)$ consisting of (isomorphism classes of) topologically trivial line bundles. 
Note that $L \in{\rm Pic}(X)$ is topologically trivial if and only if $L$ satisfies $c_1(L)=0 \in H^2(X, \mathbb{Z})$, where $c_1(L)$ stands for the first Chern class of 
$L \in {\rm Pic}(X)$. 
If $X=C$ is a smooth elliptic curve, then any topologically trivial line bundle $L \in {\rm Pic}^0(C)$ admits a structure of unitary flat line bundle (see \cite{U}). 
In particular, the monodromy of $L \in {\rm Pic}^0(C)$ along any loop in $C$ is expressed as a complex number with modulus $1$. 

\begin{definition} \label{def:diophnormal}
For $\tau \in \mathbb{H}$, 
let $C=\mathbb{C}/ \langle 1,\tau \rangle$ be a smooth elliptic curve, 
and let $\alpha$ and $\beta$ be the loops in $C$ corresponding to the line segments $[0,1]$ and $[0,\tau]$, respectively. 
Then a topologically trivial line bundle $L \in {\rm Pic}^0(C)$ on $C$ is said to satisfy the {\it Diophantine condition} if so does the pair $(p,q) \in \mathbb{R}^2$, 
where $(p,q)$ is defined from $L$, that is, $\exp(p \cdot 2\pi \sqrt{-1})$ and $\exp(q \cdot 2\pi \sqrt{-1})$ are the monodromies of $L$ along the loops $\alpha$ and $\beta$, respectively. 
\end{definition}

Now, 
assume $C_0=\mathbb{C}/ \langle 1,\tau \rangle \subset \mathbb{P}^2$ is a smooth elliptic curve embedded in the projective plane $\mathbb{P}^2$. 
Let $Z:=\set{p_1,\ldots, p_9} \subset C_0$ be nine points on $C_0$, and $S:={\rm Bl}_Z\mathbb{P}^2$ be the blow-up of $\mathbb{P}^2$ at $Z$ 
with the strict transform $C$ of $C_0$. 
In this case, the normal bundle $N_{C/S} \in {\rm Pic}(C)$ of $C$ in $S$ is isomorphic to 
$\mathcal{O}_{\mathbb{P}^2}(3)|_{C_0} \otimes \mathcal{O}_{C_0}(-p_1-\dots-p_9) \in {\rm Pic}^0(C_0) \cong {\rm Pic}^0(C)$, 
and the pair $(p,q) \in \mathbb{R}^2$ defined from $L=N_{C/S}$ (see Definition~\ref{def:diophnormal}) is given by
\[
9 p_0-\sum_{j=1}^9 p_j=q-p \cdot \tau \quad \mathrm{mod}\quad \langle 1,\tau \rangle, 
\]
where $p_0$ is an inflection point of $C_0$. Moreover, if $N_{C/S} \in {\rm Pic}^0(C)$ satisfies the Diophantine condition, then Arnol'd's theorem \cite{A} 
guarantees that there exists a analytically linearizable neighborhood of $C$ in $S$, namely, 
a tubular neighborhood of $C$ in $S$ which is biholomorphic to a neighborhood of the zero section in $N_{C/S}$. 
In other words, there exists a neighborhood of $C$ in $S$ biholomorphic to 
\begin{equation} \label{eqn:nbhdW}
W := \set{(z, w)\in \mathbb{C}^2\mid \abs{w}<R }/\sim
\end{equation}
for some $R>1$, where $\sim$ is the equivalence relation generated by
\begin{equation} \label{eqn:relationVW}
(z, w) \sim
(z+1,\ \exp(p\cdot 2\pi\sqrt{-1})\cdot w) \sim
(z+\tau,\ \exp(q\cdot 2\pi\sqrt{-1})\cdot w). 
\end{equation}
With the neighborhood $W$ at hand, we can construct a family of K3 surfaces as mentioned in the introduction. 

\begin{remark} \label{rem:leviflat}
For a given $w_0 \in \mathbb{C}$ with $0<\abs{w_0} <R$, let $F : \mathbb{C} \to W \subset S$ be a holomorphic map defined by $F(z)=[(z, w_0)]$. 
Since $(p,q)$ satisfies the Diophantine condition, the Euclidean closure of $F(\mathbb{C})$ in $S$ coincides with 
$\set{[(z, w)] \mid \abs{w}=\abs{w_0}} \subset W$, which is a real analytic hypersurface $C^\omega$-diffeomorphic to a real $3$-dimensional torus $\bS^1\times \bS^1\times \bS^1$. 
The maps $F_b$ in Theorem \ref{thm:main} can be constructed in this manner.
\end{remark}

\subsection{Holomorphic line bundles on toroidal groups}

The neighborhood $W$ given in \eqref{eqn:nbhdW} is closely related to the {\it toroidal group}. 
For $\tau \in \mathbb{H}$ and a non-torsion pair $(p,q) \in \mathbb{R}^2$, we consider 
\[
U=U_{\tau,(p,q)}:=\mathbb{C}^2_{(z, \eta)}/\Lambda \quad \text{ with } \quad 
\Lambda=\Lambda_{\tau,(p,q)}:=\left\langle 
\left(
    \begin{array}{c}
       0 \\
       1
    \end{array}
  \right), 
\left(
    \begin{array}{c}
       1 \\
       p
    \end{array}
  \right), 
\left(
    \begin{array}{c}
       \tau \\
       q
    \end{array}
  \right)
\right\rangle. 
\]
It is seen that $U$ becomes a toroidal group (see \emph{e.g.} \cite{AK}). 
On the toroidal group $U$, an important class of line bundles is the {\it theta line bundles}, given as follows. 
Let 
\[
H=\left(
    \begin{array}{cc}
       a & b \\
       \overline{b} & c
    \end{array}
  \right) \in \rM_2(\mathbb{C})
\]
be a Hermitian matrix satisfying the condition
\begin{equation} \label{eqn:condi}
\mathrm{Im}\, H(\lambda,\mu) \in \mathbb{Z} \quad (\lambda , \mu \in \Lambda), 
\end{equation}
where $H(x,y)={}^t x H \overline{y}$ for $x,y \in \mathbb{C}^2$, and let $\rho : \Lambda \to U(1)$ be a {\it semi-character} of $\mathrm{Im}\, H$, that is, it satisfies 
\[
\rho(\lambda+\mu)=\rho(\lambda)\rho(\mu)\exp\left(\pi \sqrt{-1} \mathrm{Im}\, H(\lambda,\mu)\right) \quad (\lambda , \mu \in \Lambda). 
\]
Then we define the holomorphic function $\alpha_\lambda=\alpha^{(H,\rho)}_\lambda : \mathbb{C}^2_{(z, \eta)} \to \mathbb{C}$ by
\[
\alpha_\lambda(x) := \rho(\lambda)\exp\left(\pi H(x, \lambda)+(\pi/2) H(\lambda, \lambda)\right), \quad \lambda \in \Lambda,\ x={}^t(z, \eta) \in \mathbb{C}^2. 
\]
From \eqref{eqn:condi}, the function $\alpha_\lambda(x)$ satisfies the cocycle condition
\[
\alpha_{\lambda+\mu}(x) = \alpha_{\lambda}(x+\mu) \alpha_{\mu}(x), \quad \lambda,\mu \in \Lambda,\ x \in \mathbb{C}^2, 
\]
and hence 
\[
L=L_{H,\rho}:=(\mathbb{C}_{\zeta} \times \mathbb{C}^2)/ \Lambda
\]
with 
\[
\lambda \cdot (\zeta,x):=(\alpha_{\lambda}(x) \cdot \zeta, x+\lambda), \quad \lambda \in \Lambda,\ \zeta \in \mathbb{C}_{\zeta},\  x \in \mathbb{C}^2
\]
defines a line bundle on $U$, which is called a {\it theta line bundle} on $U$. 
In our setting, note that $\lambda_2 \in \mathbb{R}$ for any ${}^t (\lambda_1,\lambda_2) \in \Lambda$. 
Hence a nowhere vanishing holomorphic function $\beta : \mathbb{C}^2 \to \mathbb{C}^*$, given by 
\[
\beta(z, \eta)=\exp (-\pi c \eta^2/2),
\]
satisfies
\[
\alpha^{(H_0,\rho)}_\lambda(x) =\beta(x+\lambda)\alpha^{(H,\rho)}_{\lambda}(x) \beta(x)^{-1} \quad (\lambda \in \Lambda, x \in \mathbb{C}^2) 
\qquad 
\text{with} \quad H_0=\left(
    \begin{array}{cc}
       a & b \\
       \overline{b} & 0
    \end{array}
  \right),
\]
which means that $L_{H,\rho}$ is holomorphically isomorphic to $L_{H_0,\rho}$. Hereafter, we assume $c=0$ and put 
\begin{equation} \label{eqn:hermitian}
H=\left(
    \begin{array}{cc}
       a & b \\
       \overline{b} & 0
    \end{array}
  \right) \in \rM_2(\mathbb{C}). 
\end{equation}
On the line bundle $L_{H,\rho}$, there is a natural metric $h=h_H$, given by 
\[
|\zeta|_{h, x}^2 := \exp(-\pi H(x, x)) |\zeta|^2, 
\]
which is well-defined because
\begin{align*}
|\alpha_\lambda(x)\cdot \zeta|_{h, x+\lambda}^2
&= |\alpha_\lambda(x)|^2\cdot \exp(-\pi H(x+\lambda, x+\lambda)) |\zeta|^2 \\
&=\exp({\rm Re}(2\pi H(x, \lambda)+\pi H(\lambda, \lambda))) \cdot \exp(-\pi H(x+\lambda, x+\lambda)) |\zeta|^2 \\
&=\exp(\pi H(x, \lambda) + \pi H(\lambda, x) +\pi H(\lambda, \lambda))\\
&\quad \cdot \exp(-\pi H(x, x)-\pi H(x, \lambda)-\pi H(\lambda, x)-\pi H(\lambda, \lambda))|\zeta|^2\\
&=\exp(-\pi H(x, x))|\zeta|^2=|\zeta|_{h, x}^2. 
\end{align*}
In particular, the curvature form of $h_H$ is given by 
\[
\Theta_{h_H} := - \partial \overline{\partial} \log h_H=
\pi \cdot (a dz\wedge d\overline{z} + bdz\wedge d\overline{\eta} + \overline{b}d\eta\wedge d\overline{z} )
\]
with $x={}^t(z,\eta) \in \mathbb{C}^2$, and $c_1(L_{H,\rho})=[\sqrt{-1} \Theta_{h_H}/2 \pi]$. 
Moreover the following result holds (see \cite{AK}). 

\begin{proposition}\label{prop:thetabdle}
Assume that $(p,q)$ satisfies the Diophantine condition. Then any line bundle $L$ on $U_{\tau,(p,q)}$ is holomorphically isomorphic to $L_{H,\rho}$ for some $(H,\rho)$. 
\end{proposition}


\subsection{Deformations of K3 surfaces and Picard numbers}

The following results are taught by Dr. Takeru Fukuoka. 

\begin{proposition}\label{prop:K3_deform_picard}
Let $P\colon \mathcal{X}\to T$ be a deformation family of K3 surfaces. 
Assume that the Kodaira--Spencer map $\rho_{\rm KS, P} \colon T_T\to R^1P_*T_{\mathcal{X}/T}$ is injective. 
Then, for almost every $t\in T$, it holds that $\rho(X_t)\leq20-\dim(T)$, where $X_t:=P^{-1}(t)$ and $\rho(X_t)$ is the Picard number of $X_t$. 
\end{proposition}

\begin{proof}
Take a base point $0\in T$ and denote by $L:=\Pi_{3, 19}$ the K3 lattice $H^2(X_0, \mathbb{Z})$. 
Fix a marking $R^2P_*\mathbb{C}_{\mathcal{X}}\cong(L_\mathbb{C})_T$, where $L_\mathbb{C}:=L\otimes\mathbb{C}$. 
Consider the map $V_\bullet\colon T\to \mathbb{P}(L_\mathbb{C})$ defined by 
$t\mapsto V_t:=H^0(X_t, K_{X_t})^{\bot}$ for each $t\in T$, where we are regarding $\mathbb{P}(L_\mathbb{C})$ as the set of hyperplanes of $L_\mathbb{C}$. 
It follows from Torelli's theorem that the map $V_\bullet$ is a locally closed embedding of $T$ into $\mathbb{P}(L_\mathbb{C})$. 
Therefore ${\rm Image}\,V_\bullet$ is a locally closed subvariety of $\mathbb{P}(L_\mathbb{C})$ of dimension $\dim(T)$. 
Define $r\colon \mathbb{P}(L_\mathbb{C})\to \mathbb{Z}$ by $r(V):={\rm rank}(L\cap V)$. 
Note that $r(V_t)={\rm rank}(H^2(X_t, \mathbb{Z})\cap (H^{1, 1}(X_t, \mathbb{C})\oplus H^{0, 2}(X_t, \mathbb{C})))=\rho(X_t)+1$ holds for each $t\in T$. 
Therefore the set $\set{t\in T\mid \rho(X_t)< 21-\dim(T)}$ can be rewritten as 
$V_\bullet^{-1}\left(({\rm Image}\,V_\bullet)\setminus \set{V\in \mathbb{P}(L_\mathbb{C})\mid r(V)\geq 22-\dim(T)}\right)$. 
By Lemma \ref{lem_countable_union_F_r} below, $\set{V\in \mathbb{P}(L_\mathbb{C})\mid r(V)\geq 22-\dim(T)}$ is a countable union of ($\dim(T)-1$)-dimensional linear subspaces of $\mathbb{P}(L_\mathbb{C})$. 
\end{proof}

\begin{lemma}\label{lem_countable_union_F_r}
Let $r\colon \mathbb{P}(L_\mathbb{C})\to \mathbb{Z}$ be as in the proof of Proposition \ref{prop:K3_deform_picard}. 
Then $F_n:=\set{V\in \mathbb{P}(L_\mathbb{C})\mid r(V)\geq n}$ is a countable union of $(21-n)$-dimensional linear subspaces of $\mathbb{P}(L_\mathbb{C})$ for each $n=0, 1, 2, \ldots, 21$. 
\end{lemma}

\begin{proof}
Set $\Lambda:=\set{M\subset L\mid M: \text{sub module},\ {\rm rank}\,M=n}$. 
For $M\in\Lambda$ and $W\in \mathbb{P}(L_\mathbb{C}/M_\mathbb{C})$, it clearly holds that $p_M^{-1}(W)\in F_n$, 
where $M_\mathbb{C}:=M\otimes \mathbb{C}$ and $p_M\colon L_\mathbb{C}\to L_\mathbb{C}/M_\mathbb{C}$ is the natural projection. 
Conversely, for each $V\in F_n$ and a sublattice $M\subset V$ of rank $n$, we have $V=p_M^{-1}(W)$ by defining $W:=V/M_\mathbb{C}\in \mathbb{P}(L_\mathbb{C}/M_\mathbb{C})$. 
Therefore we obtain the description 
\[
F_n=\bigcup_{M\in\Lambda}\set{p_M^{-1}(W)\mid W\in \mathbb{P}(L_\mathbb{C}/M_\mathbb{C})}. 
\]
As $\Lambda$ is countable and the map $p_M^{-1}(-)\colon \mathbb{P}(L_\mathbb{C}/M_\mathbb{C})\ni W\mapsto p_M^{-1}(W)\in F_n\subset \mathbb{P}(L_\mathbb{C})$ is a linear embedding for each $M$, the lemma follows. 
\end{proof}


\section{Line bundles on $W$ and $V$} \label{sec:linebundles}

For $\tau \in \mathbb{H}$, let $C=\mathbb{C}_{z} /\langle 1,\tau \rangle$ be a complex torus, 
and for a non-torsion pair $(p, q) \in\mathbb{R}^2$ and $0 \le r<R \le \infty$, let $W =W_{\tau,(p,q)}^R$ be defined in \eqref{eqn:nbhdW} and $V =V_{\tau,(p,q)}^{r,R}$ 
be defined by 
\[
V =V_{\tau,(p,q)}^{r,R}:= \set{(z, w)\in \mathbb{C}^2\mid r<\abs{w}<R }/\sim, 
\]
where $\sim$ is given by \eqref{eqn:relationVW}. 
We notice that $V$ is isomorphic to an open submanifold of the toroidal group $U=U_{\tau,(p,q)}=(\mathbb{C}_z \times \mathbb{C}_{\eta})/\Lambda$,
namely,
\[
U \supset \left(\mathbb{C}_z \times \set{ -\log R < 2 \pi \mathrm{Im}\,\eta<-\log r }\right)/\Lambda \ni [(z,\eta)] \overset{\cong}{\longmapsto} \left[\left(z,\exp(2 \pi \sqrt{-1} \eta)\right)\right] \in V
\]
with $U_{\tau,(p,q)} \cong V_{\tau,(p,q)}^{0,\infty}$, 
and $W$ is obtained from $V_{\tau,(p,q)}^{0,R}$ by adding the complex torus $C$.
Let $\pi : W \to C$ be the natural projection, given by $\pi([(z,w)])=[z]$, and denote $\pi|_{V} : V \to C$ by $\pi : V \to C$ for simplicity. 

\begin{lemma}\label{top_triv_lb}
Assume that $(p,q)$ satisfies the Diophantine condition. Then for any $L\in {\rm Pic}^0(W)$, the equality $L=\pi^*(L|_C)$ holds.
\end{lemma}

\begin{proof}
As the topologically trivial bundle $L$ satisfies $c_1(L)=0$, $L$ can be represented by some $\alpha\in H^1(W, \mathcal{O}_W)$ 
from the exact sequence $H^1(W, \mathcal{O}_W) \to {\rm Pic}(W) \overset{c_1}{\longrightarrow} H^2(W, \mathbb{Z})$. 
Hence it is enough to show that $\pi^*(\alpha|_C) =\alpha$. 

Put $\alpha=\{(W_{jk}, f_{jk})\}$, where $W_{jk}=W_j \cap W_k$ and $W_j=\pi^{-1}(U_j)\cong U_j\times \Delta$ with a Stein open covering $\{U_j\}$ of $C$. 
Moreover $f_{jk}$ can be expressed on $W_j$ as a convergent power series
\[
f_{jk}(z_j, w_j) = \sum_{n=0}^\infty f_{jk, n}(z_j)\cdot w_j^n, 
\]
where $(z_j, w_j)$ are coordinates on $W_j$ which come from $(z, w)$. 
Then it is enough to show that there are holomorphic functions $g_j\colon W_j\to \mathbb{C}$ such that 
\[
\set{(W_{jk}, \widehat{f}_{jk})} = \delta\set{(W_j, g_j)}:= \left\{ (W_{jk}, -g_j + g_k) \right\} ,
\]
where 
\[
\widehat{f}_{jk}(z_j, w_j) := f(z_j, w_j)-f(z_j, 0) = \sum_{n=1}^\infty f_{jk, n}(z_j)\cdot w_j^n. 
\]
Note that there exists a multiplicative $1$-cocycle $\{t_{jk}\}$ with $t_{jk} \in U(1)$ representing $N_{C/W}$ such that $w_k=t_{kj} \cdot w_j$ for any $j,k$. 
Since $\set{(U_{jk}, f_{jk, n})} \in H^1\left(\{U_j\}, N_{C/W}^{-n}\right)$ and $N_{C/W}$ is non-torsion, 
the $\delta$-equation 
\[
-g_{j, n}+t_{jk}^{-n}\cdot g_{k, n} = f_{jk, n}
\]
has a unique solution $g_{j, n} \colon U_j\to \mathbb{C}$ for each $n>0$. 
Furthermore the power series 
\begin{equation} \label{eqn:powerseries}
g_j(z_j, w_j)=\sum_{n=1}^\infty g_{j, n}(z_j)\cdot w_j^n
\end{equation}
 converges. 
Indeed, Ueda's lemma (see \cite[Lemma~4]{U}) says that there exists a constant $K>0$ depending only on $C$ and $\set{U_j}$ such that  for any flat line bundle $E$ over $C$ 
and for any $0$-cochain $\set{h_j} \in C^0\left(\set{U_j},\mathcal{O}(E)\right)$, the inequality
\[
d(\mathbb{I}_C,E) \cdot \norm{\set{h_j}} \le K \cdot \norm{\delta \set{h_j}}
\]
holds, where $\mathbb{I}_C$ is the holomorphically trivial line bundle on $C$, $d(\mathbb{I}_C,E)$ is the Euclidean distance of ${\rm Pic}^0(C)\cong \mathbb{C}/\langle 1, \tau\rangle$, which clearly is an invariant distance, and 
\[
\norm{\set{h_j}}:=\max_{j} \sup_{z \in U_j}\abs{h_j(z)} \quad\text{and}\quad \norm{\delta\set{h_j}}:=\max_{j,k} \sup_{z \in U_j \cap U_k}\abs{h_{jk}(z)} \quad \text{with}\ \ \set{h_{jk}}:=\delta\set{h_j}. 
\]

In our setting, since $N_{C/W}$ satisfies the Diophantine condition, there exist $A>0$ and $\alpha>0$ such that $d(\mathbb{I}_C, N_{C/W}^n)\geq A\cdot n^{-\alpha}$ holds for any $n\geq 1$. 
Cauchy's inequality shows that for any $\ell \in (0,R)$, there exists $M>0$ such that $\abs{f_{jk,n}(z_j)} \le M/\ell^n$ for any $n \ge 1$ and $z_j \in U_j \cap U_k$. 
Hence we have 
\[
\abs{g_{j,n}(z_j)} \le \frac{K}{d(\mathbb{I}_C,N_{C/W}^n)} \cdot  \max_{j,k} \sup_{z_j \in U_j \cap U_k}\abs{f_{jk}(z_j)} \le \frac{K}{A\cdot n^{-\alpha}} \cdot \frac{M}{\ell^n}
=\frac{KM}{A} \cdot \frac{n^{\alpha}}{\ell^n}, 
\]
which means that the power series \eqref{eqn:powerseries} indeed converges because $\ell \in (0,R)$ is chosen arbitrarily. 
Therefore we have $\pi^*(\alpha|_C) =\alpha$ in $H^1(W, \mathcal{O}_W)$. 
\end{proof}

\begin{remark}
The following can be proved in a similar manner by replacing a Taylor power series with a Laurent power one: for any $L\in {\rm Pic}^0(V)$, there exists an $F \in {\rm Pic}^0(C)$ such that $L= \pi^*F$, which is proved in \cite{AK} for the case where $V=U$ is a toroidal group. 
Conversely, \cite{AK} also proves the statement that if a pair $(p,q)$ {\it does not} satisfy the Diophantine condition, 
then there exists an $L\in {\rm Pic}^0(U)$ such that  $L \neq \pi^*F$ for any $F \in {\rm Pic}^0(C)$. 
\end{remark}

\begin{proposition} \label{prop:linebundleonW}
Assume that $(p,q)$ satisfies the Diophantine condition. 
Then $L=\pi^*(L|_C)$ holds for any $L\in {\rm Pic}(W)$. In particular, the restriction map ${\rm Pic}(W)\to {\rm Pic}(C)$ is an isomorphism.
\end{proposition}

\begin{proof}
As $C$ is a deformation retract of $W$, the restriction map $H^2(W, \mathbb{Z})\to H^2(C, \mathbb{Z})$ is an isomorphism. 
Hence we have $c_1(L \otimes \pi^*(L^{-1}|_C))=0$ and $L \otimes \pi^*(L^{-1}|_C)$ is topologically trivial. 
Since $(L \otimes \pi^*(L^{-1}|_C))|_C$ is a trivial bundle on $C$, one has $L=\pi^*(L|_C)$ by Lemma \ref{top_triv_lb}. 
\end{proof}

Now let us recall the three $2$-cycles 
\[
A_{\alpha\beta}=\bS_{\alpha}^1 \times \bS_{\beta}^1, \quad A_{\beta\gamma}=\bS_{\beta}^1 \times \bS_{\gamma}^1, \quad \text{and}\quad A_{\gamma\alpha}=\bS_{\alpha}^1 \times \bS_{\gamma}^1 
\]
on $V$, where, for a base point $[(0,w_0)] \in V$, $\bS_{\alpha}^1, \bS_{\beta}^1, \bS_{\gamma}^1$ are circles given by the images of 
\begin{itemize}
\item $i_{\alpha}\colon [0, 1] \ni \alpha \mapsto \left[(\alpha, \exp(\alpha p \cdot 2 \pi \sqrt{-1}) w_0)\right] \in V$, \\
\item $i_{\beta}\colon [0, 1] \ni \beta \mapsto \left[(\beta \tau, \exp(\beta q \cdot 2 \pi \sqrt{-1}) w_0)\right] \in V$, \\
\item $i_{\gamma}\colon [0, 1] \ni \gamma \mapsto \left[(0, \exp(\gamma \cdot 2 \pi \sqrt{-1}) w_0)\right] \in V$, 
\end{itemize}
respectively. Here, the orientations of $A_{\alpha\beta}, A_{\beta\gamma}, A_{\gamma\alpha}$ are defined by $d \alpha \wedge d \beta$, $d \beta \wedge d \gamma$, $d \alpha \wedge d \gamma$, respectively. 

\begin{lemma} \label{lem:intersection}
For a Hermitian matrix $H$ given in \eqref{eqn:hermitian} satisfying condition \eqref{eqn:condi} and a semi-character $\rho$ of $\,{\rm Im}\,H$, we have
\begin{enumerate}
\item $(L_{H,\rho}\cdot A_{\alpha\beta})=\mathrm{Im}\, H(x_{\beta},x_{\alpha})=a\cdot {\rm Im}\,\tau +p\cdot {\rm Im}(b \tau) -q \cdot {\rm Im}\,b$, \\
\item $(L_{H,\rho}\cdot A_{\beta\gamma}) = \mathrm{Im}\, H(x_{\gamma},x_{\beta})=-{\rm Im}(b\tau)$, \\
\item $(L_{H,\rho}\cdot A_{\gamma\alpha}) = \mathrm{Im}\, H(x_{\gamma},x_{\alpha})=-{\rm Im}\,b$, 
\end{enumerate}
where $x_{\alpha}:={}^t (1,p)$, $x_{\beta}:={}^t (\tau,q)$, and $x_{\gamma}:={}^t (0,1)$.
\end{lemma}

\begin{proof}
We will only prove the assertion (1) as the other cases can be treated in the same manner. 
Note that the class $c_1(L_{H,\rho})$ can be represented as 
\[
\frac{\sqrt{-1}}{2} \cdot (a dz\wedge d\overline{z} + bdz\wedge d\overline{\eta} + \overline{b}d\eta\wedge d\overline{z}), 
\]
where $w=\exp(\eta \cdot 2 \pi \sqrt{-1})$. By the definition of $A_{\alpha\beta}$, put $z=\alpha+\tau \beta$ and $\eta=p\alpha+q \beta$. 
Since $p, q, \alpha, \beta \in \mathbb{R}$, we have
\[
j_{\alpha\beta}^*dz \wedge d\overline{z} =d(\alpha+\tau \beta) \wedge d(\alpha+\overline{\tau} \beta)=(\overline{\tau}-\tau) d \alpha \wedge d \beta = -2 \sqrt{-1} \mathrm{Im}\,\tau d \alpha \wedge d \beta, 
\]
where $j_{\alpha\beta}\colon A_{\alpha\beta}\to V$ is the embedding induced by $i_\alpha$ and $i_\beta$. 
In a similar manner, one has 
\[
j_{\alpha\beta}^*dz\wedge d\overline{\eta} = - (p \tau-q)d \alpha \wedge d \beta, \quad j_{\alpha\beta}^*d\eta\wedge d\overline{z} = \overline{(p \tau-q)} d \alpha \wedge d \beta, 
\]
and hence
\[
j_{\alpha\beta}^*(bdz\wedge d\overline{\eta} + \overline{b}d\eta\wedge d\overline{z}) = -2 \sqrt{-1} \mathrm{Im} \left(b (p \tau-q)\right) d \alpha \wedge d \beta. 
\]
Therefore we have
\[
(L_{H,\rho}\cdot A_{\alpha\beta})=\int_{[0,1]\times [0,1]} \left(a  \mathrm{Im}\,\tau +  \mathrm{Im} (b (p \tau-q))\right) d \alpha \wedge d \beta=a  \mathrm{Im}\,\tau +  \mathrm{Im} (b (p \tau-q)). 
\]
\end{proof}

\begin{proposition} \label{prop:extentability}
Let $L\in {\rm Pic}(V)$ be a holomorphic line bundle on $V$. 
Assume that $(p,q)$ satisfies the Diophantine condition. 
Then the following are equivalent. 
\begin{enumerate}
\item There exists a holomorphic line bundle $G\in {\rm Pic}(W)$ on $W$ such that $L=G|_V$. 
\item $(L\cdot A_{\beta\gamma}) = (L\cdot A_{\gamma\alpha}) =0$.
\item The equality $b=0$ holds, where $b$ is the $(1, 2)$-element of the Hermitian matrix $H \in \rM_2(\mathbb{C})$ as \eqref{eqn:hermitian} 
satisfying the condition \eqref{eqn:condi} and $L=L_{H, \rho}$ for a semi-character $\rho$ of $\mathrm{Im}\, H$, whose existence is assured by Proposition~\emph{\ref{prop:thetabdle}}. 
\end{enumerate}
\end{proposition}

Note that the Diophantine assumption on the pair $(p, q)$ in this proposition can be dropped if one assumes that $L=L_{H, \rho}$ for some Hermitian matrix $H \in \rM_2(\mathbb{C})$ satisfying condition \eqref{eqn:condi} and $\rho$ is a semi-character of $\mathrm{Im}\, H$. 

\begin{proof}
The equivalence (2) $\Longleftrightarrow$ (3) follows from Lemma \ref{lem:intersection} and 
(1) $\Longrightarrow$ (2) holds since the circle $\bS_{\gamma}^1$ is contractible in $W$. 
The implication (3) $\Longrightarrow$ (1) follows since the factor $\alpha^{(H,\rho)}_{\lambda}(z, \eta)$ depends only on $z$ and thus $L$ is expressed as $L=\pi^{*}(L_0)$ for some $L_0 \in {\rm Pic}(C)$. 
\end{proof}


\section{Proofs of main theorems} \label{sec:proofs}

\subsection{Proof of Theorem \ref{thm:main1} $(i)$}\label{section:proof_main_i}

It follows from Proposition \ref{prop:linebundleonW} and the assumption $g_{\xi}^*\left(L^-|_{C^-}\right) \cong L^+|_{C^+}$ that 
the restrictions $L^{\pm}|_{V^{\pm}_s}$ of $L^{\pm}|_{W^{\pm}}$ are isomorphic via the biholomorphic map $f_s : V_s^+ \to V_s^-$. 
Thus, $\left(M^+_s, L^+|_{M^+_s}\right)$ and $\left(M^-_s, L^-|_{M^-_s}\right)$ are glued together to yield a holomorphic line bundle $L_s=L^+ \vee L^-$ on $X_s$. 

In order to describe the holomorphic line bundle $\mathcal{L}\to \mathcal{X}$ on $\mathcal{X}$ via the isomorphisms \eqref{eqn:isomorphism}, 
we define manifolds $\mathcal{M}^{\pm}$ and $\mathcal{V}$ by
\[
\mathcal{M}^{\pm}:= (S^{\pm} \times \Delta ) \setminus \set{ (\left[(z^{\pm},w^{\pm})\right],s) \in W^{\pm} \times \Delta \mid \abs{w^{\pm}} \le \sqrt{\abs{s}} R}
\]
and 
\[
\mathcal{V}:=\set{(z^+,w^+,w^-)\in \mathbb{C}^3 \mid \abs{w^{+}}<R,\, \abs{w^{-}}<R,\, \abs{w^+ w^-}<1}/\sim, 
\]
where $\sim$ is the equivalence relation generated by
\[
(z^+, w^+,w^-) \sim
(z^++1,\ e^{p \cdot 2\pi\sqrt{-1}}\cdot w^+,\ e^{-p \cdot 2\pi\sqrt{-1}}\cdot w^-) \sim 
(z^++\tau,\ e^{q \cdot 2\pi\sqrt{-1}}\cdot w^+,\ e^{-q \cdot 2\pi\sqrt{-1}}\cdot w^-). 
\]
Then $\mathcal{M}^{\pm}$ and $\mathcal{V}$ are glued together to yield the deformation family $\mathcal{X}$ via injective holomorphic maps $f_{\pm} : \mathcal{M}^{\pm} \supset \mathcal{V}^{\pm} \to \mathcal{V}$, where 
\[
\mathcal{V}^{\pm}:=  \set{ ([(z^{\pm},w^{\pm})],s) \in W^{\pm} \times\Delta \mid \sqrt{\abs{s}} R<\abs{w^{\pm}} < R} \subset \mathcal{M}^{\pm}
\]
and 
\[
f_{+}\left(([(z^+,w^+)],s)\right)=[(z^+,w^+,s/w^+)], \quad f_{-}\left(([(z^-,w^-)],s)\right)=[(g_{\xi}^{-1}(z^-),s/w^-,w^-)]. 
\]
The restriction of $\mathcal{X} \to \Delta$ on $\mathcal{M}^{\pm}$ is the natural projection $\mathcal{M}^{\pm} \to \Delta$, 
while that on $\mathcal{V}$ is given by $[(z^+,w^+,w^-)] \mapsto w^+ \cdot w^-$. 
Moreover, it should be noted that there are natural projections $\varphi_{\pm} : \mathcal{M}^{\pm} \to S^{\pm}$ and $\varphi : \mathcal{V} \to C^+$ given by $\varphi([(z^+,w^+,w^-)])=[z^+]$. 
Then a holomorphic line bundle $\mathcal{L} \to \mathcal{X}$ is defined by the pullbacks $\varphi_{\pm}^* (L^{\pm})$ on $\mathcal{M}^{\pm}$ and $\varphi^*(L^+|_{C^+})$ on $\mathcal{V}$. 
We notice that the line bundle $\mathcal{L} \to \mathcal{X}$ is well-defined since 
the line bundles $f_+^* \varphi^*(L^+|_{C^+})$ and $f_-^* \varphi^*(g_{\xi}^*(L^-|_{C^-}))$ are the same as the restrictions $\mathcal{L}|_{\mathcal{V}^{+}}$ and $\mathcal{L}|_{\mathcal{V}^{-}}$ respectively, 
by virtue of Proposition \ref{prop:linebundleonW} and the assumption $g_{\xi}^*(L^-|_{C^-}) \cong L^+|_{C^+}$. 
\qed

\subsection{Idea of proof of Theorem \ref{thm:main1} $(ii)$}
 
Let $X=X_s$ be a K3 surface obtained by gluing $M^+=M_s^+$ and $M^-=M_s^-$, and $L^{\pm}$ be an ample line bundle on $S^{\pm}$. 
In order to show Theorem \ref{thm:main1} $(ii)$, we will construct a $C^{\infty}$-Hermitian metric on $L:=L_s=L^+ \vee L^-$ with positive curvature in the following manner for fixed $0<R_1<R_2<R$: 

\begin{description}
\item[Step 1] Construct a $C^{\infty}$-Hermitian metric $h_{\pm}$ on $L^\pm$ such that: 
\begin{itemize}
\item[---] $h_{\pm}$ can be glued to define a $C^{\infty}$-Hermitian metric $h$ on $L$ (if $0<\abs{s}<\ve_0$), 
\item[---] the Chern curvature of $h_{\pm}$ is semi-positive: $\sqrt{-1}\Theta_{h_\pm}\geq 0$, 
\item[---] $\sqrt{-1}\Theta_{h_\pm} > 0$ holds on $S^{\pm} \setminus \set{\abs{w^{\pm}} \le R_1}$, and 
\item[---] $\sqrt{-1}\Theta_{h_\pm}(\partial/\partial z^{\pm},\partial/\partial z^{\pm}) > 0$ holds on $S^{\pm}$. 
\end{itemize}
\item[Step 2] Construct a $C^{\infty}$ function $\psi^{\pm}$ on $S^{\pm} \setminus C^{\pm}$ such that: 
\begin{itemize}
\item[---] $\psi^{\pm}$ can be glued to define a $C^{\infty}$ function $\psi$ on $X$, 
\item[---] $\psi^{\pm}$ is psh on $M^{\pm} \setminus \set{R_2 \le \abs{w^{\pm}} \le R}$: $\sqrt{-1} \partial \bar\partial \psi^{\pm}|_{ M^{\pm} \setminus \set{R_2 \le \abs{w^{\pm}} \le R}} \ge0$, 
\item[---] $\psi^{\pm}|_{W^{\pm}}$ depends only on $\abs{w^{\pm}}$, and 
\item[---] $\sqrt{-1}\partial \bar\partial \psi^{\pm}(\partial/\partial w^{\pm},\partial/\partial w^{\pm}) > 0$ holds on $\set{\abs{w^{\pm}}<R_2 }$. 
\end{itemize}
\item[Step 3] For $0<c\ll 1$, $h\cdot e^{-c\psi}$ is a desired metric on $L$ with positive Chern curvature $\sqrt{-1}\Theta_{h}+ c \sqrt{-1}\partial \bar\partial \psi >0$. 
\end{description}

In our construction, $h_\pm\cdot e^{-c\psi^\pm}$ is a $C^{\infty}$-Hermitian metric on $L^\pm|_{S^\pm\setminus C^\pm}$ with positive Chern curvature 
such that $h_\pm\cdot e^{-c\psi^\pm} \sim (\log |w^{\pm}|)^2$ as $w^{\pm} \to 0$. 
Moreover, $\omega^\pm :=\sqrt{-1}\Theta_{h_\pm} +c\sqrt{-1}\ddbar \psi^\pm \in c_1\left(L^\pm|_{S^\pm\setminus C^\pm}\right)$ gives a complete K\"ahler metric on $S^\pm\setminus C^\pm$, and 
on a neighborhood $\set{\abs{w^\pm}<\sqrt{\ve_0}R}$ of $C^\pm$, the form $\omega^\pm$ is expressed as 
\[
\omega^\pm|_{\set{\abs{w^\pm}<\sqrt{\ve_0}R}} = \frac{\pi (L^{\pm}\cdot C^{\pm})}{\mathrm{Im}\,\tau}\cdot \sqrt{-1}dz^\pm\wedge d\overline{z}^\pm + 2c\cdot \frac{\sqrt{-1}dw^\pm\wedge d\overline{w}^\pm}{\abs{w^\pm}^2}. 
\]

\subsection{Proof of Theorem \ref{thm:main1} $(ii)$}

Let $S$ be the blow-up of $\mathbb{P}^2$ at nine points, and $C \subset S$  be an elliptic curve in $|K_{S}^{-1}|$ such that 
$N_{C/S} \in {\rm Pic}^0(C)$ satisfies the Diophantine condition. Then Arnol'd's theorem says that there is an analytically linearizable neighborhood $W \subset S$ of $C$. 
By shrinking $W$ if necessary, we may assume that $W$ is isomorphic to $W_{\tau,(p,q)}^{R}$ for some $R>0$, $\tau \in \mathbb{H}$ and $(p,q) \in \mathbb{R}^2$ 
satisfying the Diophantine condition, and let $\pi : W \to C$ be the projection given in Section~\ref{sec:linebundles}. 

Let $L \in {\rm Pic }(S)$ be an ample line bundle, which implies that there exists $n \in \mathbb{N}$ such that $L^n\otimes [-C]$ is very ample, 
and let $g_1, g_2, \ldots, g_N$ be a basis of $H^0(S, L^n\otimes [-C])$, which are regarded as sections of $L^n$ with zeros along $C$. 
Then the {\it singular} Hermitian metric $h_L$ on $L$ is defined by 
\[
\langle \xi, \eta\rangle_{h_L, x} := \frac{\xi\cdot \overline{\eta}}{\left(\abs{g_1(x)}^2+\abs{g_2(x)}^2+\cdots + \abs{g_N(x)}^2\right)^{\frac{1}{n}}}, \quad\text{where}\ \ \xi, \eta\in L|_x. 
\]
The metric $h_L$ has a pole along $C$ and 
its restriction $h_L|_{S\setminus C}$ induces a $C^{\infty}$-metric on $S \setminus C$ with positive curvature form $\sqrt{-1}\Theta_{h_L}|_{S \setminus C}>0$. 
Moreover let $h_C$ be a $C^{\infty}$-metric on $L|_W$ satisfying $\sqrt{-1}\Theta_{h_C}= b \sqrt{-1}dz\wedge d\overline{z}$ for $b:=\pi (L\cdot C)/\mathrm{Im}\,\tau >0$. 

Fix $0<R_1 <R_2<R$. Then we define a metric $h$ on $L$ by 
\[
h^{-1} := \begin{cases}
{\rm Regularized Max} (h_L^{-1}, \ve\cdot \pi^*h_{C}^{-1}) & \text{on}\ W \\
h_L^{-1} & \text{on}\ S\setminus\overline{W}
\end{cases}
\]
where $\ve>0$ and ${\rm Regularized Max}\colon\mathbb{R}^2\to \mathbb{R}$ is the regularized maximum function (see \cite[Chapter~\romfigure{1}, Lemma~5.18]{D}). 
Note that, by choosing $\ve>0$ sufficiently small, one may assume that $h=h_L$ holds on $\set{[(z, w)]\in W\mid R_1< \abs{w}}$, which ensures the smoothness of $h$. 
Then $\sqrt{-1}\Theta_h\geq 0$, since the local weight function $\vp$ of $h$ satisfies
\[
\vp = {\rm Regularized Max}(\vp_L, \vp_C-\log\ve), 
\]
where $\vp_L$ and $\vp_C$ are the local weight functions of $h_L$ and $h_C$, respectively. 
By the construction of $h$, there exists a positive constant $\ve_0$ such that 
$h=\ve^{-1} \cdot \pi^*h_{C}$ holds on $\set{\abs{w}< \sqrt{\ve_0} R }$. 
By shrinking $\ve_0$ if necessary, we may assume $\sqrt{\ve_0} R<R_1$. 
For $s \in \Delta$ with $|s|< \ve_0$, let $\lambda=\lambda_s : \mathbb{R}_{>0} \to \mathbb{R}$ be a $C^{\infty}$-function satisfying the conditions
\[\begin{cases}
\lambda(t)=\left(\log (t^2/\abs{s}) \right)^2  &\text{if}\ 0<t<R_2,\\
\lambda(t)\equiv\text{constant}&\text{if}\ t \geq R,
\end{cases}\]
and $\psi=\psi_s : S \setminus C \to \mathbb{R}$ be the $C^{\infty}$-function defined by
\[
\psi(p):=
\begin{cases}
\lambda(\abs{w}) & \forall\, p=(z,w) \in W \setminus C \\
\lambda(R) & \forall\, p \notin W. 
\end{cases}
\]
It is easy to see that $\ddbar\psi=0$ outside $\set{ \abs{w} \le R }$ and 
$\ddbar \psi=2\cdot dw\wedge d\overline{w}/\abs{w}^2$
on $\set{0<\abs{w} < R_2}$. 
Finally, we choose $c>0$ so that 
\[
\sqrt{-1}\Theta_{h_L}+c\sqrt{-1}\ddbar \psi>0
\]
on the compact subset $\set{R_2 \le \abs{w} \le R}$. 
Here note that such a $c>0$ exists since $\sqrt{-1}\Theta_{h_L}$ is strictly positive on $\set{R_2 \le \abs{w} \le R} \subset S \setminus C$. 

We consider the metric $h \cdot e^{-c \psi}$ on $S \setminus C$. Our assumption on $c>0$ says that 
\[\sqrt{-1}\Theta_{h \cdot e^{-c \psi}} = \sqrt{-1} \Theta_{h_L} + c \sqrt{-1}\ddbar \psi>0\] outside $\set{\abs{w}<R_2}$. 
Moreover, $h \cdot e^{-c \psi}$ has positive curvature also on $\set{0<\abs{w}<R_2}$, since it holds 
\[
\sqrt{-1}\Theta_{h_L \cdot e^{-c \psi}} > \sqrt{-1} \Theta_{h_L}>0, \quad 
\sqrt{-1}\Theta_{\ve^{-1} \cdot \pi^*h_C \cdot e^{-c \psi}}=b \sqrt{-1} dz \wedge \overline{z}+ c \frac{\sqrt{-1}dw\wedge d\overline{w}}{\abs{w}^2}>0
\]
and 
\[
(h\cdot e^{-c \psi})^{-1} = {\rm Regularized Max} \left((h_L\cdot e^{-c \psi})^{-1}, (\ve^{-1}\cdot \pi^*h_{C}\cdot e^{-c \psi})^{-1} \right) 
\]
on $\set{0<\abs{w}<R_2}$ (see \cite[Chapter~\romfigure{1}, Lemma~5.18(e)]{D}). 
Therefore the curvature of $h \cdot e^{-c \psi}$ is positive on $S \setminus C$. 

Now we consider two pairs $(S^{\pm},C^{\pm})$ of surfaces $S^{\pm}$ and curves $C^{\pm} \subset S^{\pm}$ given in the introduction, which admit analytically linearizable neighborhoods $W^{\pm} \subset S^{\pm}$ of $C^{\pm}$, and assume that $W^{\pm}$ are regarded as subspaces $\set{[(z^{\pm},w^{\pm})] \mid \abs{w^{\pm}} <R}$ of  toroidal groups. 
Moreover let $L^{\pm}$ be ample line bundles with $(L^+\cdot C^+)=(L^-\cdot C^-)$ and $g_{\xi} : C^+ \to C^-$ be an isomorphism with $g_\xi^*\left(L^-|_{C^-}\right)=L^+|_{C^+}$. 
In what follows we abuse the notation to denote $g_\xi$ simply by $g$. 
Then the above argument shows that there exist $C^{\infty}$-metrics $h_{\pm} \cdot e^{-c \psi^{\pm}}$ on $S^{\pm} \setminus C^{\pm}$ such that 
$\sqrt{-1}\Theta_{h_{\pm} \cdot e^{-c \psi^{\pm}}}>0$ on $S^{\pm} \setminus C^{\pm}$ and 
\[
h_{\pm} =\ve^{-1} \cdot \pi_{\pm}^*h_{C_{\pm}}, \quad \psi^{\pm}(z^{\pm}, w^{\pm}) = \left(\log \frac{\abs{w^{\pm}}^2}{\abs{s}} \right)^2
\]
on $\set{0<\abs{w^{\pm}}<\sqrt{\abs{s}}R \ (<\sqrt{\ve_0} R<R_1)}$. 
As our K3 surface $X_s$ is given by gluing two surfaces
\[M_s^{\pm}=S^{\pm} \setminus \set{\abs{w^{\pm}} \le\sqrt{\abs{s}}/R }\]
via the map 
$(z^+, w^+) \mapsto (z^-, w^-)=(g(z^+), s/w^+)$, 
it follows from Proposition \ref{prop:linebundleonW} that $h_{\pm}$ can be glued together and become a global $C^{\infty}$-Hermitian metric on $L_s = L^+\vee L^-$. 
Moreover, on $\set{ \sqrt{\abs{s}}/R<\abs{w^+}<\sqrt{\abs{s}}R }$, we have 
$\psi^+(z^+, w^+) = \left(\log \abs{w^+}^2/\abs{s}\right)^2$ and 
\[
\psi^-\left(g(z^+), \frac{s}{w^+}\right) = \left(\log \frac{\abs{s/w^+}^2}{\abs{s}}\right)^2
= \left(-\log \frac{\abs{w^+}^2}{\abs{s}} \right)^2
= \psi^+(z^+, w^+) ,
\]
which means that $\psi^{\pm}$ can be glued together and become a global $C^{\infty}$-function $\psi$ on $X_s$. 
Therefore $h_{\pm} \cdot e^{-c \psi^{\pm}}$ yield a $C^{\infty}$-metric on $X_s$ with positive definite curvature form. 
\qed


\subsection{Proof of Theorem \ref{thm:main1} $(iii)$}
The equivalence (a) $\Longleftrightarrow$ (c) follows from Proposition~\ref{prop:extentability}, and the implications (a) $\Longrightarrow$ (b) follows from Theorem~\ref{thm:main1} (i). 
In what follows we show (b) $\Longrightarrow$ (c). 
Take a line bundle $\mathcal{L}\to \mathcal{X}$ as in (b) and consider the function 
$h\colon \Delta\to \mathbb{Z}$ defined by 
\[
h(t) := \left(\mathcal{L}|_{M_t^+}\cdot\ A_{\beta\gamma}\right), 
\]
where we are regarding $A_{\beta\gamma}$ as a cycle of $M_t^+$. 
As $\mathcal{M}^+\to \Delta$ is a submersion, $h$ is continuous. Thus $h$ is a constant function. 
Therefore, in order to show that $(L\cdot A_{\beta\gamma}) (=h(s))$ is equal to zero, it is sufficient to show that $h(0)=0$, which follows from Proposition \ref{prop:extentability} since $\mathcal{L}|_{M_0^+}$ coincides with the restriction of the line bundle $(\mathcal{L}|_{X_0})|_{S^+}$ to $M_0^+$. 
The equation $(L\cdot A_{\gamma\alpha})=0$ can be shown in the same manner. 
\qed

\subsection{Proof of Theorem \ref{thm:main}}

Our construction of K3 surfaces has $19$ complex dimensional degrees of freedom if we allow the variation of $\xi$ \cite{KU}. 
Indeed, for a fixed pair $(p,q) \in \mathbb{R}^2$ satisfying the Diophantine condition,  we have the following parameters: 
\begin{enumerate}[label=(\Roman*)]
\item $1$ parameter $\tau\in\mathbb{H}$ determining the elliptic curve $C^+ \cong C^-$, 
\item $16$ parameters $\{p_1^{\pm},\dots, p_8^{\pm}\}$ determining  the centers of the blow-ups $\pi^{\pm}$ (here $p_9^+$ and $p_9^-$ are fixed from the conditions (a) and (b) in the introduction), 
\item $1$ parameter $\xi \in \mathbb{C}$ determining the isomorphism $g_{\xi} : C^+ \to C^-$, and
\item $1$ parameter $s \in \Delta \setminus \{0\}$ determining the gluing function $f_s : V_s^+ \to V_s^-$. 
\end{enumerate}
Note that there always exist ample line bundles $L^\pm\to S^\pm$ with $(L^+\cdot C^+)=(L^-\cdot C^-)$. 
If such ample line bundles $L^{\pm}$ are fixed, then $\xi$ is determined uniquely up to modulo $\langle 1, \tau\rangle$ from the condition $g_{\xi}^*\left(L^-|_{C^-}\right) \cong L^+|_{C^+}$, 
and depends holomorphically on the parameters given in (\romfigure{1}) and (\romfigure{2}) (see also the relation \eqref{eqn:xifix}). 
Moreover, for any $s \in \Delta \setminus \{0\}$ with sufficiently small $\abs{s}\ll 1$, the K3 surface $X_s$ admits an ample line bundle $L_s =L^+ \vee L^-$ by Theorem \ref{thm:main1} $(ii)$. 
Hence we have an $18$ dimensional family of projective K3 surfaces, whose Kodaira-Spencer map is injective by \cite[Theorem 1.1]{KU}. Moreover it follows from \cite{KU} that there exists 
a holomorphic immersion $F_b\colon \mathbb{C}\to X_b$ mentioned in Theorem \ref{thm:main} (see also Remark \ref{rem:leviflat}). 
Finally among the family, almost every fiber is a non-Kummer K3 surface 
since if follows from Proposition \ref{prop:K3_deform_picard} that almost every fiber $X_s$ has  the Picard number $\rho(X_s) \le 2$. 
\qed

\section{Calculation of the Chern class $c_1(L)$} \label{sec:calculation}

Let $S^{\pm}$  be surfaces obtained from the blow-ups $\pi^{\pm} : S^{\pm} \to \mathbb{P}^2$ of 
the projective plane $\mathbb{P}^2$ at nine points $\{p_1^{\pm}, \ldots, p_9^{\pm}\}$ with smooth elliptic curves $C^{\pm} \in \abs{K_{S^\pm}^{-1}}$. 
In our assumption $(S^{\pm},C^{\pm})$ satisfy Conditions (a) and (b) given in the introduction. 
Moreover let $L^{\pm}$ be holomorphic line bundles on $S^{\pm}$. 
In this section, we compute the Chern class $c_1(L)$ in the lattice $H^2(X, \mathbb{Z})\cong H_2(X, \mathbb{Z})$, where $X$ is a K3 surface given by the gluing construction and $L = L^+ \vee L^-$ is the line bundle on $X$
(see the introduction). 

First we notice that the second homology group of $S^{\pm}$ is expressed as 
\[
H_2(S^{\pm}, \mathbb{Z}) \cong H^2(S^{\pm},\mathbb{Z}) \cong \mathrm{Pic}(S^{\pm})= \langle H^{\pm},E_1^{\pm},\ldots, E_9^{\pm} \rangle, 
\]
where $E_\nu^\pm$ is (the class of) the exceptional divisor in $S^\pm$ which is the preimage of $p_\nu^\pm$ for $\nu=1, 2, \ldots, 9$, 
and $H^\pm$ is (the class of) the preimage of a line in $\mathbb{P}^2$ by the blow-up $\pi^\pm\colon S^\pm\to \mathbb{P}^2$. 
In the homology group $H_2(S^{\pm}, \mathbb{Z})$, the elliptic curve $C^{\pm}$ is expressed as 
\[
C^{\pm}=3H^{\pm}-\sum_{j=1}^9 E_j^{\pm}. 
\]
We also notice that the points $p_1^{\pm}, \ldots, p_9^{\pm}$ lie in the elliptic curve $C_0^{\pm}:=\pi^{\pm}(C^{\pm})$. 
Then fix isomorphisms
\[
C_0^+ \cong C_0^- \cong \mathbb{C}/\langle 1,\tau \rangle
\]
and also fix an inflection point $p_0^{\pm}$ so that 
\[
9p_0^{\pm}-\sum_{j=1}^9 p_j^{\pm}= \pm \mu\quad \text{mod}\ \langle 1, \tau\rangle,
\]
where $\mu:= q-p \cdot \tau$ and the points $p_j^{\pm} \in \mathbb{C}$ $(j=0,\ldots,9)$ are regarded as complex numbers (see Subsection~\ref{subsec:nbhdellcurve}). 
By choosing the complex number corresponding to the point $p_0^\pm$ appropriately, we may assume that 
\begin{equation} \label{eqn:torrelation}
9p_0^{\pm}-\sum_{j=1}^9 p_j^{\pm}= \pm \mu,
\end{equation}
actually holds. 
In what follows we assume that $g(p_0^+)=p_0^-$ by changing $g$ if necessary. 
For $j \neq k \in \set{0,1,\ldots,9}$, let $\Gamma_{jk}^{\pm} \subset C^{\pm}$ be the lift of an arc in $C^\pm_0$ connecting $p_{j}^{\pm}$ and $p_{k}^{\pm}$. 

Now we give the definitions of the generators \eqref{eqn:generator} (see also \cite{KU}). 
The $2$-cycles $A_{\alpha \beta}$, $A_{\beta \gamma}$, $A_{\gamma \alpha}$ 
are already defined in the introduction. 
In order to define the 2-cycle $B_{\bullet}$ for $\bullet\in\set{\alpha, \beta, \gamma}$, we first notice that $M_s^{\pm}$ are simply connected. 
Thus, there exist topological discs $D_{\bullet}^{\pm} \subset M_s^{\pm}$ such that 
$\partial D_{\bullet}^{\pm}=\pm \bS_{\bullet}^1$ hold, 
where $\bS_{\bullet}^1 \subset V_s$, which are given in the introduction, are regarded as $1$-cycles of $V_s^{\pm} \subset M_{s}^{\pm}$. 
Then $B_{\bullet}$ is defined by $B_{\bullet}=D_{\bullet}^+ \cup_{\bS_{\bullet}^1} (-D_{\bullet}^-)$, 
that is, the patch of $D_{\bullet}^+$ and $-D_{\bullet}^-$ through $\bS_{\bullet}^1$. 
In order to define the $2$-cycles $C_{\bullet}^{\pm}$, we prepare the tube $T_{jk}^{\pm}$ given by 
$T_{jk}^{\pm}:=\mathrm{pr}_{\pm}^{-1}(\Gamma_{jk}^{\pm}) \subset \set{\abs{w^{\pm}}=\sqrt{\abs{s}}}$, where 
\[
\mathrm{pr}_{\pm} : \set{[(z^\pm, w^\pm)]\in W^\pm\mid \abs{w^\pm}= \sqrt{\abs{s}}} \ni [(z^\pm, w^\pm)] \longmapsto [z^{\pm}] \in C^{\pm}
\]
is a natural projection. Then for $\nu=1,\ldots,7$, the $2$-cycle $C_{\nu, \nu+1}^{\pm}$ is defined by the connected sum 
$(\pm E_{\nu}^{\pm}) \# (\mp E_{\nu+1}^{\pm})$ of 
$\pm E_{\nu}^{\pm}$ and $\mp E_{\nu+1}^{\pm}$ given by connecting them through the tube $T_{\nu, \nu+1}^{\pm}$. In a similar manner, the $2$-cycle $C_{678}^{\pm}$ is defined by the connected sum 
\[C_{678}^{\pm}:=(\mp H^{\pm}) \# (\pm E_{6}^{\pm}) \# (\pm E_{7}^{\pm}) \# (\pm E_{8}^{\pm})\]
of $\mp H^{\pm}$, $\pm E_{6}^{\pm}$, $\pm E_{7}^{\pm}$, $\pm E_{8}^{\pm}$ given by connecting them through the tubes 
$T_{0 6}^{\pm}$, $T_{0 7}^{\pm}$, $T_{0 8}^{\pm}$. 
In particular, $C_{\bullet}^{\pm}$ is represented as
\[
C_{12}^{\pm}=\pm(E_1^{\pm}-E_2^{\pm}),\  \ldots , \ C_{78}^{\pm}=\pm(E_7^{\pm}-E_8^{\pm}), 
\ C_{678}^{\pm}=\pm(-H^{\pm}+E_6^{\pm}+E_7^{\pm}+E_8^{\pm})
\]
in $H_2(S^{\pm}, \mathbb{Z})$.
It should be noted that $C_{\bullet}^{\pm}$ lies in $M_s^{\pm}$. 
Moreover, $H_2(S^{\pm}, \mathbb{Z})$ admits an orthogonal decomposition
\[
H_2(S^{\pm}, \mathbb{C}) = \langle C^{\pm}, E_9^{\pm} \rangle \oplus \mathcal{C}^{\pm}
\]
with respect to the intersection product, where $\mathcal{C}^\pm$ is given by $\mathcal{C}^\pm:=\langle C^\pm_{12}, C^\pm_{23},\ldots,C^\pm_{78},C^\pm_{678}\rangle$, 
and any element $q^{\pm} \in H_2(S^{\pm}, \mathbb{C})=H_2(S^{\pm}, \mathbb{Z}) \otimes \mathbb{C}$ admits an expression
\begin{equation}\label{eqn:q_expression}
q^{\pm}=q_0^{\pm} H^{\pm}-\sum_{j=1}^9 q_j^{\pm} E_j^{\pm} = \Bigl(3q_0^{\pm}-\sum_{j=1}^8 q_j^{\pm}\Bigr) C^{\pm} 
+ \Bigl(3q_0^{\pm}-\sum_{j=1}^9 q_j^{\pm}\Bigr) E_9^{\pm}  + q^{\pm}|_{\mathcal{C}^{\pm}}. 
\end{equation}

Next let us consider a K3 surface $X=X_s$ given by the gluing construction. 
It is seen (\emph{cf.} \cite{KU}) that the second homology group of $X$ is given by the orthogonal decomposition
\[
H_2(X, \mathbb{Z})=\Pi_{3,19} \cong \langle A_{\alpha\beta}, B_\gamma\rangle\oplus \langle A_{\beta\gamma}, B_\alpha\rangle\oplus \langle A_{\gamma\alpha}, B_\beta\rangle\oplus \mathcal{C}^+ \oplus \mathcal{C}^-. 
\]
with respect to the intersection product. Here note that
\begin{align*}
&(A_{\alpha\beta}\cdot A_{\alpha\beta})=(A_{\beta\gamma}\cdot A_{\beta\gamma})=(A_{\gamma\alpha}\cdot A_{\gamma\alpha})=0, \\
&(B_\gamma\cdot B_\gamma)=(B_\alpha\cdot B_\alpha)=(B_\beta\cdot B_\beta)=-2, \\
&(A_{\alpha\beta}\cdot B_\gamma)= (A_{\beta\gamma}\cdot B_\alpha)=( A_{\gamma\alpha}\cdot B_\beta)=1. 
\end{align*}
The K3 surface $X$ admits a nowhere vanishing holomorphic $2$-form $\sigma$, which is expressed as
\[
\sigma = (2 \mu + c_9^-)A_{\alpha\beta} + \mu B_\gamma
+xA_{\beta\gamma} + \tau B_\alpha
+yA_{\gamma\alpha} + B_\beta
+\sum c^+_\bullet C^+_\bullet 
+\sum c^-_\bullet C^-_\bullet 
\]
in $H_2(X,\mathbb{C})$ by multiplying a constant to $\sigma$ if necessary, 
where $x=x(s)$ and $y=y(s)$ are constants and $c_{\bullet}^{\pm}$ is given by
\[
c_{\bullet}^{\pm}=\int_{\Gamma_{\bullet}^{\pm}} dz^{\pm}
\]
with $\Gamma_9^- \subset C^{-}$ being the lift of an arc in $C^-_0$ connecting $p_9^-$ and $g_{\xi}(p_9^+)$. 
Hence, one has
\begin{equation} \label{eqn:coeffcientc}
c^{\pm}_{12} =\pm(p_1^{\pm}-p_2^{\pm}), \ \ldots ,\  c^{\pm}_{78} =\pm(p_7^{\pm}-p_8^{\pm}), \ c^{\pm}_{678} =\pm(-3 p_0^{\pm}+p_6^{\pm}+p_7^{\pm}+p_8^{\pm}), \ \ \text{and}\ \ c_9^-=g_{\xi}(p_9^+)-p_9^- 
\end{equation}
if one selects the arcs appropriately. 
\begin{proposition} \label{prop:charnexpression}
The Chern class $c_1(L)$ in $H^2(X, \mathbb{Z})\cong H_2(X, \mathbb{Z})$ is expressed as 
\[
c_1(L) = (2b+n_9^++n_9^-)A_{\alpha\beta} + b B_\gamma
+L^+|_{\mathcal{C}^+}+L^-|_{\mathcal{C}^-},
\]
where $b:=(L^+\cdot C^+) =(L^-\cdot C^-)$ and $n_9^{\pm}:=(L^{\pm}\cdot E_9^{\pm})$. 
\end{proposition}

\begin{proof} We put
\[
c_1(L) = \widehat{a}_{\alpha\beta}A_{\alpha\beta} + \widehat{b}_\gamma B_\gamma
+\widehat{a}_{\beta\gamma}A_{\beta\gamma} + \widehat{b}_\alpha B_\alpha
+\widehat{a}_{\gamma\alpha}A_{\gamma\alpha} + \widehat{b}_\beta B_\beta
+\sum\widehat{c}^+_\bullet C^+_\bullet 
+\sum\widehat{c}^-_\bullet C^-_\bullet. 
\]
First the coefficients $\widehat{c}^{\pm}_\bullet$ are determined from $L^{\pm}|_{\mathcal{C}^{\pm}}$ since the cycles $C^{\pm}_\bullet$ in $X_s$ are also regarded as the ones 
in $M_s^{\pm} \subset S^{\pm}$. Next it follows from Theorem \ref{thm:main1} $(iii)$ that $(L\cdot A_{\beta\gamma})=(L\cdot A_{\gamma\alpha})=0$, which implies that $\widehat{b}_\alpha=\widehat{b}_\beta=0$. 
Moreover, the cycle $A_{\alpha\beta}$ may be regarded as $C^{\pm}$ in $S^{\pm}$, which means that $(L\cdot A_{\alpha\beta})=(L^\pm\cdot C^\pm)=b$ and thus 
$\widehat{b}_\gamma = b$. 

Finally we will determine the coefficients $\widehat{a}_\bullet$. To this end, we put 
\[
p^{\pm}:= 3 p_0^{\pm}H^{\pm}- \sum_{j=1}^9 p_j^{\pm} E_j^{\pm} \in H_2(S^{\pm}, \mathbb{C}). 
\]
Then Condition \eqref{eqn:torrelation} shows that $(p^{\pm}\cdot C^{\pm})=\pm \mu$
and the relation \eqref{eqn:coeffcientc} means that  $\sigma|_{\mathcal{C}^{\pm}}= \pm p^{\pm}|_{\mathcal{C}^{\pm}}$. 
Thus it follows from Equation \eqref{eqn:q_expression} that 
\begin{equation} \label{eqn:pexpression} 
p^{\pm}=  \left(9p_0^{\pm}-\sum_{j=1}^8 p_j^{\pm}\right) C^{\pm} + \left(9p_0^{\pm}-\sum_{j=1}^9 p_j^{\pm}\right) E_9^{\pm} + p^{\pm}|_{\mathcal{C^{\pm}}} 
=  (\pm \mu + p_9^{\pm}) C^{\pm} \pm \mu E_9^{\pm} \pm \sigma|_{\mathcal{C}^{\pm}}. 
\end{equation} 
Moreover, 
by the condition $g_\xi^{*}\left(L^-|_{C^-}\right)=L^+|_{C^+}$, we may assume that $\xi \in \mathbb{C}$ is given by  
\begin{equation} \label{eqn:xifix}
\xi = \frac{1}{b} \left((p^-\cdot L^-)-(p^+\cdot L^+)\right). 
\end{equation}
Thus we have
\begin{equation} \label{eqn:subrel1}
c_9^-=g_{\xi}(p_9^+)-p_9^-=(p_9^++\xi)-p_9^-=\frac{1}{b} \left((p^-\cdot L^-)-(p^+\cdot L^+)\right)+(p_9^+-p_9^-). 
\end{equation}
Equation \eqref{eqn:pexpression} shows that
\begin{equation} \label{eqn:subrel2}
(\sigma|_{\mathcal{C}^{\pm}}\cdot L^{\pm})= \pm (p^{\pm}\cdot L^{\pm}) - (\mu {\pm} p_9^{\pm})(C^{\pm}\cdot L^{\pm}) - \mu(E_9^{\pm}\cdot L^{\pm})
= \pm (p^{\pm}\cdot L^{\pm}) {\mp} b p_9^{\pm} -b \mu  -n_9^{\pm} \mu, 
\end{equation}
and Equations \eqref{eqn:subrel1} and \eqref{eqn:subrel2} show that 
\begin{align*}
0 = (\sigma.\cdot L) 
&=\left(((2 \mu + c_9^-)A_{\alpha\beta} + \mu B_\gamma)\cdot (\widehat{a}_{\alpha\beta}A_{\alpha\beta} + b B_\gamma)\right)+\left((xA_{\beta\gamma} + \tau B_\alpha)\cdot \widehat{a}_{\beta\gamma}A_{\beta\gamma}\right)\\
&\qquad +\left((yA_{\gamma\alpha} + B_\beta)\cdot \widehat{a}_{\gamma\alpha}A_{\gamma\alpha}\right)
+\left(\sigma|_{\mathcal{C}^+}\cdot L^+\right)+\left(\sigma|_{\mathcal{C}^-}\cdot L^-\right)\\
&= \left(\mu\widehat{a}_{\alpha\beta}+\tau \widehat{a}_{\beta\gamma}+\widehat{a}_{\gamma\alpha}\right)+b c_9^-
+(\sigma|_{\mathcal{C}^+}\cdot L^+)+(\sigma|_{\mathcal{C}^-}\cdot L^-)
=  \left(\mu\widehat{a}_{\alpha\beta}+\tau \widehat{a}_{\beta\gamma}+\widehat{a}_{\gamma\alpha}\right) - (2b+n_9^++n_9^-) \mu. 
\end{align*}
Since $\widehat{a}_{\alpha\beta}, \widehat{a}_{\beta\gamma}, \widehat{a}_{\gamma\alpha}, b, n_9^+,n_9^-\in \mathbb{Z}$ and $(1,\tau,\mu)$ are 
independent over $\mathbb{Q}$ by the Diophantine condition for the pair $(p, q)$, we have $\widehat{a}_{\alpha\beta}=2b+n_9^++n_9^-$ and $\widehat{a}_{\beta\gamma}=\widehat{a}_{\gamma\alpha}=0$. 
\end{proof}



\begin{thebibliography}{Dem12++}

\bibitem[AK01]{AK}
Y.~Abe and K.~Kopfermann, \emph{Toroidal groups. Line bundles, cohomology and quasi-abelian varieties}, Lecture Notes in Mathematics {\bf 1759}, Springer-Verlag, Berlin, 2001.
 
\bibitem[Arn77]{A}
V.~I.~Arnol'd, \emph{Bifurcations of invariant manifolds of differential equations and normal forms in neighborhoods of elliptic curves}, Funkcional Anal. i Prilozen. \textbf{10-4} (1976), 1--12. English translation: Functional Anal. Appl. \textbf{10-4} (1977), 249--257.

\bibitem[Dem12]{D}
J.-P. Demailly, \emph{Complex Analytic and Differential Geometry}, (2012). Available from \href{https://www-fourier.ujf-grenoble.fr/~demailly/manuscripts/agbook.pdf}{https://www-fourier.ujf-grenoble.fr/~demailly/manuscripts/agbook.pdf}. 

\bibitem[KU19]{KU}
T.~Koike and T.~Uehara, \emph{A gluing construction of K3 surfaces}, preprint \arXiv{1903.01444} (2019).

\bibitem[Ued83]{U}
T.~Ueda, \emph{On the neighborhood of a compact complex curve with topologically trivial normal bundle}, J.~Math. Kyoto Univ. {\bf 22} (1983), 583--607. 

\end{thebibliography}
\end{document}